\documentclass[12pt]{article}
\usepackage{}
\usepackage{amssymb}
\usepackage{amsmath}
\usepackage[usenames]{color}
\usepackage{mathrsfs}
\usepackage{amsfonts}
\usepackage{amssymb,amsmath}
\usepackage{CJK}
\usepackage{cite}
\usepackage{cases}
\usepackage{amsthm}

\def\sc{\scriptstyle}

\def\dis{\displaystyle}
\def\ul{\underline}\def\OPLUS#1{\raisebox{-8pt}{$\stackrel{\mbox{\large$\dis\oplus$}}{\sc #1}$}}
\def\bs{\backslash}\def\si{\sigma}\def\th{\theta}
\def\cl{\centerline}

\def\vs{\vspace*}

\def\Z{\mathbb{Z}}
\def\C{\mathbb{C}}

\numberwithin{equation}{section}
\newtheorem{theo*}{Theorem}
\newtheorem{rema*}{Remark}
\newtheorem{theo}{Theorem}[section]

\newtheorem{lemm}[theo]{Lemma}
\newtheorem{prop}[theo]{Proposition}

\newtheorem{clai}{Claim}

\def\ptl{\partial}\def\si{\sigma}\def\th{\theta}
\begin{document}
\begin{CJK*}{GBK}{song}
\def\ell{l}

\begin{center}
{\bf\large Automorphism groups of Witt algebras\,$^*$}
\footnote {$^*\,$Supported by NSF grant no. 11371278, 11431010, 
the Fundamental Research Funds for the Central Universities of China, Innovation Program of Shanghai Municipal Education Commission and  Program for Young Excellent Talents in Tongji University.

\ $\ ^\dag\,$Correspondence: Y.~Su (Email: ycsu@tongji.edu.cn)

}
\end{center}

\cl{ Jianzhi Han, Yucai Su$\,^{\dag}$}

\cl{\small Department of Mathematics, Tongji University, Shanghai
200092, China}

\vs{8pt}

{\small
\parskip .005 truein
\baselineskip 3pt \lineskip 3pt

\noindent{{\bf Abstract:} The automorphism groups ${\rm Aut}\,A_n$ and ${\rm Aut\,}W_n$ of  the polynomial algebra $A_n=$\break $\C[x_1,x_2,\cdots, x_n]$ and the rank $n$ Witt algebra $W_n={\rm Der\,}A_n$ are  studied in this paper. It is well-known that  ${\rm Aut\,}A_n$ for $n\ge3$ and ${\rm Aut\,}W_n$ for $n\ge2$ are open. In the present paper,
by characterizing  the semigroup ${\rm End\,}W_n\bs\{0\}$ of nonzero endomorphisms of $W_n$ via the semigroup of the so-called Jacobi tuples,
we establish an isomorphism between ${\rm Aut}\,A_n$ and ${\rm Aut\,}W_n$ for any positive integer $n$. In particular, this enables us to work out the automorphism group ${\rm Aut\,}W_2$
of $W_2$. \vs{5pt}

}

}
\parskip .001 truein\baselineskip 6pt \lineskip 6pt
\section{Introduction}
\setcounter{equation}{0}
With a history of 100 years \cite{C}, the (one-sided rank $n$) Witt algebras $W_n={\rm Der\,}A_n$ (the derivation algebras of the polynomial algebras $A_n:=\C[x_1,...,x_n]$ of $n$ variables for all $n\ge1$) are the first known examples of infinite-dimensional simple Lie algebras. However, the determination of automorphism groups ${\rm Aut\,}W_n$ of $W_n$ is a long outstanding open problem (even for case $n=2$). It is well-known (e.g., \cite{MTU,SU1,SuXu1,SuXu2,SuXuZhang,SuZhao1,SuZhao2,SuZhaoZhu})
that
automorphism groups of Lie algebras constitute an important part in the structure theory of Lie algebras.
For the case of the two-sided Witt algebra $W_n^{\pm}={\rm Der\,}A_n^{\pm}$ (the derivation algebra of the  Laurent polynomial algebra $A_n^{\pm}:=\C[x_1^{\pm1},...,x_n^{\pm1}]$),
the problem of determining the automorphism group ${\rm Aut\,}W_n^{\pm}$ of $W_n^{\pm}$ is much easier (e.g., \cite{DZ,SuXu1,SuXu2,SuXuZhang,Zhao}), as any automorphism $\sigma\in{\rm Aut\,}W_n^{\pm}$ must fix the set ${\mathscr F}_{W_n^{\pm}}$ of the ${\rm ad}$-locally finite elements of
$W_n^{\pm}$ and in this case ${\mathscr F}_{W_n^{\pm}}$ turns out to be the vector space ${\mathscr F}_{W_n^{\pm}}=\oplus_{i=1}^n\C x_i\frac{\partial}{\partial x_i}$.
In sharp contrast to $W_n^{\pm}$, the set ${\mathscr F}_{W_n}$ of the ${\rm ad}$-locally finite elements of
$W_n$ is unachievable.

The distinguished Jacobi conjecture posed by Keller in 1939 says  that if $f_1,f_2,..., f_n\!\in\! A_n$ are
$n$ polynomials on $n$ variables such that the corresponding Jacobi determinant $J(f_1,...,f_n)$ $:={\rm Det}\big(\frac{\partial f_i}{\partial x_j}\big){}_{1\le i,j\le n}\in\C^\times:=\C\bs\{0\}$ is a nonzero complex number (in this case, the $n$-tuple $(f_1,...,f_n)$ is referred to as
a Jacobi tuple in the present paper),  then $f_1,f_2,..., f_n$ are generators of $A_n$, namely, $A_n=\C[f_1,f_2,...,f_n]$. Many interesting results would follow if this conjecture holds. Unfortunately, over seven decades' endeavor made by many mathematicians (e.g., \cite{J,R,M1,SU,V1,V2,W,Y}),
it 
is still an open problem.
Obviously, the Jacobi conjecture is equivalent to the statement that every endomorphism of $A_n$ sending the generating tuple
$(x_1,...,x_n)$ to a Jacobi tuple is
an automorphism. Thus the Jacobi conjecture
is closely related to the automorphism group ${\rm Aut\,}A_n$ of $A_n$.
The group ${\rm Aut\,}A_n$
is clear in case $n\leq2$ (cf. \cite{N}), but for $n\geq 3$ this is yet undetermined.
Obviously, there are three types automorphisms: $s_i,\tau_a,\psi_p$ for $1\le i\le n-1$, $a\in\C^\times$, $p\in\Z^{\ge0}$,
where $s_i$ is the automorphism which
switches $x_i$ and $x_{i+1}$ and fixes other $x_j$'s, while $\tau_a$ is the automorphism which
sends $x_1$ to $ax_1$ and fixes other $x_j$'s, and
$\psi_p$  is the automorphism 
which
sends $x_2$ to $x_2+x_1^p$ and fixes other $x_j$'s.
The subgroup ${\rm Ta\,}A_n$ of ${\rm Aut\,}A_n$ generated by these three types automorphisms is the group of tame automorphisms, and the elements of ${\rm Aut\,}A_n\bs{\rm Ta\,}A_n$ are called wild automorphisms. It is well-known that there are no wild automorphisms of $A_2$. The first example of a wild automorphism
is the Nagata automorphism $\sigma_1$ of $A_3$ given in \cite{N} (and proved to be wild in \cite{ShU1,ShU2}) as follows:

$$\si_1(x_1)\!=\!x_1\!-\!2(x_2^2\!+\!x_1x_3)x_2\!-\!(x_2^2\!+\!x_1x_3)^2x_3,\ \ \si_1(x_2)\!=\!x_2\!+\!(x_2^2\!+\!x_1x_3)x_3, \ \  \si_1(x_3)\!=\!x_3.
$$

%

The Jacobi conjecture is also closely related to another conjecture posed in \cite[Conjecture 1]{Zhao} (referred to as the Witt algebra's conjecture for easy reference) which states that any nonzero endomorphism of $W_n$ is an automorphism (or equivalently, any nonzero endomorphism of $W_n$
is surjective), namely, ${\rm Aut\,}W_n={\rm End\,}W_n\bs \{0\}$. In fact, it was proved in \cite[Theorem 4.1]{Zhao}  that
the Witt algebra's conjecture implies the Jacobi conjecture. From this, one can expect that the determination of ${\rm Aut\,}W_n$ is a highly nontrivial problem.

In the present paper, by embedding the Witt algebra $W_n$ for any $n\ge1$ into the derivation algebra $\bar W_n={\rm Der\,}\bar A_n$ of
the field $\bar A_n=\C(x_1,...,x_n)$ of rational functions in $n$ variables (regarding $\bar A_n$ as an algebra over $\C$), we characterize the semigroup ${\rm End\,}W_n\bs \{0\}$ via the set ${\rm JT}_n$ of Jacobi tuples of $A_n$.
This provides us a way to prove an equivalence between the Jacobi conjecture and the Witt algebra's conjecture, and to
establish an isomorphism between ${\rm Aut\,}W_n$ and ${\rm Aut\,}A_n$. The later result in turn enables us to work out the automorphism group ${\rm Aut\,}W_2$ of $W_2$.

To summarize our main results, we first give a semigroup structure on ${\rm JT}_n$ by defining for  $f=(f_1,...,f_n),$ $g=(g_1,...,g_n)\in {\rm JT}_n$,
\begin{equation}\label{JT-n-prod}
f\cdot g=h, \mbox{ where $h=(h_1,...,h_n)$ with } h_i=g_i(f_1,...,f_n)\mbox{ for }1\le i\le n,
\end{equation}where $g_i(f_1,...,f_n)=g_i|_{(x_1,\cdots,x_n)=(f_1,\cdots,f_n)}.$
 By the chain rule of partial derivatives, we see that the resulting tuple $h$ is indeed in ${\rm JT}_n$, and
obtain a semigroup ${\rm JT}_n$ under the multiplication $``\cdot"$ defined in \eqref{JT-n-prod}.

Let $f=(f_1, f_2,\cdots, f_n)\in{\rm JT}_n$ be a Jacobi tuple, and assume
$J(f_1,f_2,\cdots, f_n)=c\in\C^\times$.
Let  $\sigma_f $ be the linear map of $W_n$ by defining for $k_i\in\Z^{\ge0}$ and $1\le j\le n$,
\begin{equation}\label{sigma-f}
\sigma_f(x_1^{k_1}x_2^{k_2}\cdots x_n^{k_n}\partial_j)=f_1^{k_1}f_2^{k_2}\cdots f_n^{k_n}\th_j,\mbox{ where }\partial_j=\frac{\partial}{\partial x_j},\ \th_j=\frac1c\mbox{$\sum\limits_{\ell=1}^n$}M_{\ell j}\partial_\ell,
\end{equation}
and $M_{\ell j}$ is the $(\ell,j)$-cofactor of the Jacobi matrix $M:=\big(\frac{\partial f_i}{\partial x_j}\big){}_{1\le i,j\le n}$.
One can easily verify that $\th_j(f_i)=\delta_{i,j}$ for $1\le i,j\le n,$ from this it is easy to check that $\sigma_f$ is a nonzero endomorphism of the Lie algebra $W_n$. Thus we obtain a semigroup homomorphism
\begin{equation}\label{semi-hom1}
\xi:{\rm JT}_n\to{\rm End\,}W_n\bs\{0\}\mbox{ sending }
f\mapsto\si_f.
\end{equation}
Let $\tau\in{\rm Aut\,}A_n$. Then we have a Jacobi tuple $f_\tau:=(\tau(x_1),...,\tau(x_n))\in{\rm JT}_n$, thus $\tau$ corresponds to a nonzero endomorphism $\si_{f_\tau}\in{\rm End\,}W_n\bs\{0\}$, and we obtain a semigroup homomorphism
\begin{equation}\label{semi-hom2}
\zeta:{\rm Aut\,}A_n\to{\rm End\,}W_n\bs\{0\}\mbox{ sending }
\tau\mapsto\si_{f_\tau}.
\end{equation}

Now we can summarize our main results as follows.

\begin{theo}\label{main-theo}\begin{itemize}
\item[\rm(1)] The Jacobi conjecture is equivalent to the Witt algebra's conjecture.
\item[\rm(2)] The map in \eqref{semi-hom1} is a semigroup isomorphism $\xi:{\rm JT}_n\cong{\rm End\,}W_n\bs\{0\}$.
\item[\rm(3)] The map in \eqref{semi-hom2} induces a group isomorphism $\zeta:{\rm Aut\,}A_n\cong{\rm Aut\,}W_n$.
\item[\rm(4)] The group ${\rm Aut\,}W_2$ is generated by $s,\tau_a,\psi_p$ for $a\in\C^\times,$ $p\in\Z^{\ge0}$, where,
\begin{eqnarray*}
&\!\!\!\!\!\!\!\!\!&s(x_1^ix_2^j\partial_1+x_1^kx_2^\ell\partial_2)=
x_2^ix_1^j\partial_2+x_2^kx_1^\ell\partial_1,\\
&\!\!\!\!\!\!\!\!\!&\tau_a(x_1^ix_2^j\partial_1+x_1^kx_2^\ell\partial_2)
=a^{i}x_1^ix_2^j\partial_1+a^kx_1^kx_2^\ell\partial_2,\\
&\!\!\!\!\!\!\!\!\!&\psi_p(x_1^ix_2^j\partial_1+x_1^kx_2^\ell\partial_2)=
x_1^i(x_2+x_1^p)^j(\partial_1-x_1^{p-1}\partial_2)+x_1^k(x_2+x_1^p)^\ell\partial_2,
\end{eqnarray*} for $i,j,k,\ell\in\Z^{\ge0}$.
\end{itemize}
\end{theo}

Finally we remark that the isomorphism in Theorem \ref{main-theo}\,(3) may provide a possible way
to study automorphisms of $A_n$ using the theory of Lie algebras. This is also our goal in a sequel.
%
%
%

\section{Some lemmas}
\setcounter{equation}{0} Let $n$ be a positive integer (we assume $n\ge2$). Denote by $\underline n$, $\Z^{\ge0}$ and  $\C^\times$ the set $\{1, 2,\cdots, n\}$,  the set of all non-negative integers and the set of non-zero complex numbers, respectively.
Let $A_n=\C[x_1, x_2,\cdots,x_n]$ be the polynomial algebra of $n$ variables $x_1, x_2, \cdots, x_n$ over complex field $\C$. Denote $W_n={\rm Der}\,A_n$, the derivation algebra of $A_n$.

It is well known that $W_n$ is the  free $A_n$-module of rank $n$ with
basis $\{\ptl_i\,|\,i\in\underline n\}$: \begin{equation*}W_n=\OPLUS{i\in\ul n} A_n\ptl_i=\left.\left\{\mbox{$\sum\limits_{i=1}^n$}P_i\partial_{i}\,\right|\, P_i\in A_n\right\},\mbox{ \ where \ }\partial_i=\frac{\partial}{\partial x_i}.\end{equation*} Let $\bar A_n=\C(x_1, x_2, \cdots, x_n)$ be the quotient field of $A_n$, and $\bar W_n={\rm Der}\,\bar A_n\,$  the corresponding derivation algebra of $\bar A_n$ (regarding $\bar A_n$ as a $\C$-algebra). Then obviously, $\bar W_n$ is the $n$-dimensional $\bar A_n$-vector space with basis $\{\ptl_i\,|\,i\in\underline n\}$: \begin{equation*}\bar W_n=\OPLUS{i\in\ul n} \bar A_n\ptl_i=\left.\left\{\mbox{$\sum\limits_{i=1}^n$}P_i\partial_{i}\,\right|\, P_i\in \bar A_n\right\}\mbox{, \ and \ }W_n\subset \bar W_n.\end{equation*}

Note that the space $\bar W_n\oplus\bar A_n$ is a Lie subalgebra of the Weyl type Lie algebra $\bar{\mathscr W}_n$, where $\bar{\mathscr W}_n$ is the Lie algebra consisting of all differential operators on $\bar A_n$.
In particular, for any $a_1,a_2\in\bar A_n,D_1,D_2\in\bar W_n$, one has
\begin{equation}\label{eq-formula}
[a_1D_1,a_2D_2]=[a_1D_1,a_2]D_2+a_2[a_1D_1,D_2]=a_1D_1(a_2)D_2-a_2D_2(a_1)D_1+a_1a_2[D_1,D_2].\end{equation}

Let $\sigma\in {\rm End}\,W_n\bs\{0\}$ (the set of nonzero endomorphisms of
the  Lie algebra $W_n$). Then Ker$\,\sigma=0$ as the ideal generated by a single nonzero element in ${\rm Ker\,}\sigma$ would be $W_n$ itself. Denote
\begin{equation}\label{th-i}\mbox{$\th_i=\si(\ptl_i)\in W_n$ \ for \ $i\in\ul n$.}\end{equation}
The following is the technical lemma in obtaining our main results.
\begin{lemm}\label{lem-linear-inde}
 The elements $\th_1,...,\th_n$ are $\bar A_n$-linear independent.
\end{lemm}
\begin{proof}
Suppose conversely that there exists $I_0\subsetneq \underline n$ such that $\{\th_i\mid i\in I_0\}$ forms a maximal $\bar A_n$-linearly independent subset of $\{\th_i\mid i\in\underline n\}$. Choose any $i_1\in\underline n\bs I_0$, and assume that \begin{equation}\label{eq-th_i}
\th_{i_1}=\mbox{$\sum\limits_{j\in I_0}$}b_j\th_j\mbox{ \ for some $b_{j}\in \bar A_n$.}
\end{equation}
Let $k\in\Z^{\ge0}$ and denote $D_k=\frac1{k+2}\si(x_{i_1}^{k+2}\ptl_{i_1})$. We have
\begin{equation}\label{eq-[dj]}[\th_j,D_k]=\si\Big(\Big[\ptl_j,\frac1{k+2}x_{i_1}^{k+2}\ptl_{i_1}\Big]\Big)=0\mbox{ \ for any \ }j\in I_0.\end{equation}
Applying $\si$ to $x_{i_1}^{k+1}\ptl_{i_1}=[\ptl_{i_1},\frac1{k+2}x_{i_1}^{k+2}\ptl_{i_1}]$, by \eqref{eq-formula}--\eqref{eq-[dj]},
we obtain
\begin{equation}\label{eq}
\sigma(x_{i_1}^{k+1}\partial_{i_1})=[\th_{i_1}, D_k]=\mbox{$\sum\limits_{j\in I_0}$}[b_j\th_j, D_k]=\mbox{$\sum\limits_{j\in I_0}$}c_{jk }\th_j,
\mbox{ \  where $c_{jk}=-D_k(b_j)\in\bar A_n$.}\end{equation}
Using this and the fact that $[\th_i,\th_j]=0$ for $i,j\in\ul n$,  we have
\begin{equation}\label{eq++1}
0=\sigma([\ptl_\ell,x_{i_1}^{k+1}\partial_{i_1}])=\mbox{$\sum\limits_{j\in I_0}$}
\left[\th_\ell,c_{jk }\th_j\right]
=\mbox{$\sum\limits_{j\in I_0}$}\th_\ell(c_{jk})\th_j\mbox{ \ for any $\ell\in I_0$}.\end{equation}
This together with
the $\bar A_n$-linear independence of $\{\th_i\mid i\in I_0\}$ implies that
$\th_\ell(c_{jk})=0$ for all $j, \ell\in I_0$ and $k\in\Z^{\ge0}$. From this and \eqref{eq-formula}, we obtain
$$(k_2-k_1)\sigma(x_{i_1}^{k_1+k_2+1}\partial_{i_1})=[\sigma(x_{i_1}^{k_1+1}\partial_{i_1}),  \sigma(x_{i_1}^{k_2+1}\partial_{i_1})]=0\mbox{ \ for any $k_1, k_2\in\Z^{\ge0}$},$$  contradicting the fact that ${\rm Ker}\,\sigma=0$.
\end{proof}

\begin{lemm}\label{lemm--2}For  any $i,j\in \underline n$ and $k=1,2$, there exists $a_i\in\bar A_n$ such that
 $\sigma(x_i^k\partial_j)=a_i^k\th_j.$
\end{lemm}
\begin{proof}

For any $k\in\Z^{\ge0}$, assume that $\sigma(x_i^k\partial_j)=\sum_{l=1}^n a_{ijl}^{(k)}\th_l$ for some $a_{ijl}^{(k)}\in \bar A_n$. For $i,j,m\in\ul n$ and $1\le k\in\Z^{\geq0}$, we have
\begin{eqnarray*}
-\delta_{i,m}\mbox{$\sum\limits_{l=1}^n$} ka_{ijl}^{(k-1)}\th_l&=&-\delta_{i,m}k\sigma(x_i^{k-1}\partial_j)=\sigma([x^k_i\partial_j,\partial_m])\\ &=&[\sigma(x^k_i\partial_j), \sigma(\partial_m)]=\mbox{$\sum\limits_{l=1}^n$}[a_{ijl}^{(k)}\th_l,\th_m]\\ &=&-\mbox{$\sum\limits_{l=1}^n$}\th_m(a^{(k)}_{ijl})\th_l
\end{eqnarray*} by \eqref{eq-formula} and $[\th_i,\th_j]=0$ for $i,j\in\underline n$.
 Hence  by Lemma \ref{lem-linear-inde},  \begin{equation}\label{eq--1}\th_m(a_{ijl}^{(k)})=\delta_{im}ka_{ijl}^{(k-1)}
\mbox{ \ for all $i,j,l,m\in \underline n$ and $1\le k\in\Z^{\ge0}$.}\end{equation}  In particular,
\begin{equation}\label{d(a_{ijl}}
\th_m(a_{ijl}^{(1)})=\delta_{im}\delta_{jl}\mbox{ \ for all } i,j,l,m\in\underline n,
\end{equation} since  $a^{(0)}_{ijl}=\delta_{jl}$ (the Kronecker delta). For simplicity, denote $a_{ijl}=a_{ijl}^{(1)}$ for any $i,j,l\in\underline n$.
\begin{clai}
\label{claim1}We have $\sigma(x_i\partial_j)=a_{i}\th_j$ and $\th_j(a_i)=\delta_{ij}$ for $i,j\in\underline n$, $\mbox{where }a_i
=a_{ill}\mbox{ for all }l\in\ul n.$
\end{clai}

Using \eqref{eq-formula} and \eqref{d(a_{ijl}}, for arbitrary   $i_1,i_2,j_1,j_2\in\underline n$ we have
\begin{eqnarray*}
&&\mbox{$\sum\limits_{l=1}^n$}(\delta_{j_1i_2}a_{i_1j_2l}-\delta_{j_2i_1}a_{i_2j_1l})\th_l\\ &=&\delta_{j_1i_2}\sigma(x_{i_1}\partial_{j_2})
-\delta_{j_2i_1}\sigma(x_{i_2}\partial_{j_1})=\sigma([x_{i_1}\partial_{j_1},x_{i_2}\partial_{j_2}])\\ &=&[\sigma(x_{i_1}\partial_{j_1}), \sigma(x_{i_2}\partial_{j_2})]
=\mbox{$\sum\limits_{l_1,l_2=1}^n$}
[a_{i_1j_1l_1}\th_{l_1},a_{i_2j_2l_2}\th_{l_2}]\\
&=&\mbox{$\sum\limits_{l_1,l_2=1}^n$}\big(a_{i_1j_1l_1}\th_{l_1}(a_{i_2j_2l_2})\th_{l_2}
-a_{i_2j_2l_2}\th_{l_2}(a_{i_1j_1l_1})\th_{l_1}\big)\\
&=&\mbox{$\sum\limits_{l_1,l_2=1}^n$}(\delta_{l_1i_2}\delta_{j_2l_2}a_{i_1j_1l_1}\th_{l_2}
-\delta_{l_2i_1}\delta_{j_1l_1}a_{i_2j_2l_2}\th_{l_1})\\
&=&a_{i_1j_1i_2}\th_{j_2}-a_{i_2j_2i_1}\th_{j_1}.
\end{eqnarray*}
By Lemma \ref{lem-linear-inde} again, \begin{equation}\label{eq-1-1}\delta_{j_1i_2}a_{i_1j_2l}-\delta_{j_2i_1}a_{i_2j_1l}=\delta_{j_2l}a_{i_1j_1i_2}-\delta_{j_1l}a_{i_2j_2i_1} \quad {\rm for\ any\ } i_1,i_2,j_1,j_2,l\in\underline n. \end{equation}
Setting $j_1=j_2=j$ in \eqref{eq-1-1}, one has $$\delta_{ji_2}a_{i_1jl}-\delta_{ji_1}a_{i_2jl}=\delta_{jl}(a_{i_1ji_2}-a_{i_2ji_1})\quad {\rm for\ any\ } i_1,i_2,j,l\in\underline n,$$ which is equivalent to
\begin{equation}\label{eq-1}
\delta_{ji_2}a_{i_1jl}-\delta_{ji_1}a_{i_2jl}=0\quad {\rm for\ any\ } i_1,i_2,j\neq l\in\underline n
\end{equation}
and \begin{equation}\label{eq-2}
\delta_{ji_2}a_{i_1jj}-\delta_{ji_1}a_{i_2jj}=a_{i_1ji_2}-a_{i_2ji_1}\quad {\rm for\ any\ } i_1,i_2,j\in\underline n.
\end{equation}
Taking $i_2=j$  in \eqref{eq-1} gives $a_{i_1jl}=0$ for $i_1\neq j$ and $l\neq j$; taking $i_1=j$ in \eqref{eq-2} gives $a_{jji_2}=0$ for $i_2\neq j$. It follows
that
\begin{equation}\label{eq_{ijl}}
a_{ijl}=0\quad  {\rm for\ any }\ i, j\neq l\in\underline n.
\end{equation}
On the other hand,  it follows from  \eqref{eq-1-1} for the case $l=j_2\neq j_1$ that \begin{equation}\delta_{j_1i_2}a_{i_1j_2j_2}-\delta_{j_2i_1}a_{i_2j_1j_2}= a_{i_1j_1i_2},  \end{equation}
which together with \eqref{eq_{ijl}} gives \begin{equation}\label{a_{i_1jj}}a_{i_1j_1j_1}=a_{i_1j_2j_2}  \quad {\rm for }\ i, j_1\neq j_2\in\underline n.\end{equation}
Hence by \eqref{eq_{ijl}} and \eqref{a_{i_1jj}}, the expression of $\sigma(x_i\partial_j)$ can be rewritten as \begin{equation}\label{sigma(x_ipartial_j)}\sigma(x_i\partial_j)=a_{i}\th_j\mbox{, \ where }a_i
=a_{ij_1j_1}\mbox{ for any }j_1\in\ul n, \end{equation} and whence \eqref{d(a_{ijl}} becomes
 \begin{equation}\label{th_j(a_i)=}\th_j(a_i)=\delta_{ij}\quad {\rm for\ any\ } i,j\in\underline n.\end{equation} So the Claim \ref{claim1} is true.

\begin{clai}\label{claim2}We have
$\sigma(x_i^2\partial_j)=a_i^2\th_j$ for all $i,j\in\underline n.$
\end{clai}
 By \eqref{eq--1} and \eqref{sigma(x_ipartial_j)}, \begin{equation}\label{eq--1k=2}\th_m(a_{ijl}^{(2)})=2\delta_{im}\delta_{jl}a_{i}\quad {\rm for\ all}\ i,j,l,m\in
 \underline n
.\end{equation}
It follows from \eqref{eq-formula} and   \eqref{sigma(x_ipartial_j)}--\eqref{eq--1k=2} that\begin{eqnarray*}
&&a_{ij_1i}^{(2)}\th_{j_2}-2\delta_{j_2i}a_{i}^2\th_{j_1}\\
&=&\mbox{$\sum\limits_{l=1}^n$}\left(a_{ij_1l}^{(2)}\th_{l}(a_{i})\th_{j_2}-a_{i}\th_{j_2}(a_{ij_1l}^{(2)})
\th_{l}\right)=\mbox{$\sum\limits_{l=1}^n$}[a_{ij_1l_1}^{(2)}\th_{l},a_{i}\th_{j_2}]\\ &=&[\sigma(x_{i}^2\partial_{j_1}), \sigma(x_{i}\partial_{j_2}) ]=\sigma([x_{i}^2\partial_{j_1},x_{i}\partial_{j_2}])=\delta_{j_1i}\sigma(x_{i}^2\partial_{j_2})-2\delta_{j_2i}\sigma(x_{i}^2\partial_{j_1})\\ &=&\mbox{$\sum\limits_{l=1}^n$}(\delta_{j_1i}a_{ij_2l}^{(2)}-2\delta_{j_2i}a_{ij_1l}^{(2)})\th_l.
\end{eqnarray*}
Then by Lemma \ref{lem-linear-inde}, \begin{equation}\label{eq-main}\delta_{j_1i}a_{ij_2l}^{(2)}
-2\delta_{j_2i}a_{ij_1l}^{(2)}=\delta_{j_2l}a_{ij_1i}^{(2)}-2\delta_{j_1 l}\delta_{j_2,i}a_{i}^2\quad{\rm for\ all\ } i,j_1,j_2,l\in\underline n.\end{equation}
 Thus for $l\neq j_2$ and $l\neq j_1$, we have $\delta_{j_1i}a_{ij_2l}^{(2)}-2\delta_{j_2,i}a_{ij_1l}^{(2)}=0$, which implies $a_{j_1j_2l}^{(2)}=2\delta_{j_2j_1}a_{j_1j_1l}^{(2)}$. Hence,
\begin{eqnarray}\label{a_{j_1j_2l=}0}
a^{(2)}_{j_1j_2l}=0\quad {\rm for}\ l\neq j_1\ {\rm and}\ l\neq j_2.
\end{eqnarray}
In case $j_1=j_2=j$, by \eqref{eq-main} we have $-\delta_{ji}a_{ijl}^{(2)}=\delta_{jl}
(a_{iji}^{(2)}-2\delta_{ji}a_{i}^2)$, which gives rise to
\begin{equation}\label{a^2_{iji}=0}
a^{(2)}_{iji}=0 \quad{\rm for}\ i\neq j,
\end{equation}
and \begin{equation}\label{a^2_{iii}=}
a^{(2)}_{iii}=a_i^2\quad {\rm for\ any }\ i.
\end{equation}
It follows from taking  $l=j_1\neq j_2$ in \eqref{eq-main} that  $$\delta_{j_1i}a_{ij_2j_1}^{(2)}-2\delta_{j_2i}a_{ij_1j_1}^{(2)}=-2\delta_{j_2i}a_{i}^2,$$ from which by setting $i=j_2$
we obtain \begin{equation}\label{a^2_{j_2j_1j_1}}
a_{j_2j_1j_1}^{(2)}=a_{j_2}^2\quad {\rm for}\ j_1\neq j_2.
\end{equation}

Now let us collect some useful datum  to deduce the relation promised in Claim \ref{claim2}. By \eqref{a_{j_1j_2l=}0} and \eqref{a^2_{iji}=0} one can see that
\begin{equation}
a^{(2)}_{ijl}=0\quad {\rm for\ all }\ i,j\neq l\in\underline n.
\end{equation}
It immediately follows from \eqref{a^2_{iii}=} and \eqref{a^2_{j_2j_1j_1}} that \begin{equation}
a_{ijj}^{(2)}=a_i^2\quad{\rm for\ all}\ i, j\in\underline n.
\end{equation}
Combining the above two equations gives $a^{(2)}_{ijl}=\delta_{jl}a_i^2$, and therefore \begin{equation}
\sigma(x_i^2\partial_j)=a_i^2\th_j\quad{\rm for\ all\ } i,j\in\underline n.
\end{equation}This completes  the proofs of Claim \ref{claim2} and the lemma.\end{proof}

\begin{lemm}\label{lemm-homorphism}Let $a_i$ be as in Lemma $\ref{lemm--2}$.
We have in fact $a_i\in A_n$ for all $i\in \underline n$, and $$\sigma(x_1^{k_1}x_2^{k_2}\cdots x_n^{k_n}\partial_j)=a_1^{k_1}a_2^{k_2}\cdots a_n^{k_n}\th_j\mbox{ \ for any $j\in\underline n, k_i\in\Z^{\ge0}$.}$$
\end{lemm}
\begin{proof}
First we assert $\sigma(x_i^k\partial_j)=a_i^k\th_j$ for any $i,j\in \underline n$ and $k\in\Z^{\ge0}$. We proceed by induction on $k$. By Lemma \ref{lemm--2}, this is true for $k\leq 2$. In particular, by  \eqref{eq-formula}, \eqref{eq--1} and \eqref{th_j(a_i)=} we have \begin{eqnarray}\label{x_ix_l}\sigma(x_lx_i\partial_i)=\frac12\sigma([x_l\partial_i, x_i^2\partial_i])=\frac12[a_l\th_i,a_i^2\th_i]=a_la_i\th_i\quad {\rm for\ any\ } l\neq i\in\underline n.\end{eqnarray}  Suppose that this assertion holds  for the case $k$. Let us see  the  case $k+1$. By  inductive assumption,  we have
\begin{equation*}
\sigma(x_i^{k+1}\partial_j)=\frac12\sigma([x_i^k\partial_i, x_i^2\partial_j])=\frac12[a_i^k\th_i, a_i^2\th_j]=\frac12a_i^k\th_i(a_i^2)\th_j=a_i^{k+1}\th_j\quad {\rm for}\ i\neq j,
\end{equation*} and in case $i=j$, we can always choose $l\ne i$ (since we assume $n\ge2$) such that
\begin{eqnarray*}\sigma(x_i^{k+1}\partial_i)&=&\sigma([x_i^k\partial_l, x_lx_i\partial_i]+k[x_i^k\partial_i, x_ix_l\partial_l])\\ &=&[a_i^k\th_l, a_la_i\th_i]+k[a_i^k\th_i, a_ia_l\th_l]
=a_i^{k+1}\th_i
\end{eqnarray*} by \eqref{x_ix_l}.
So in either case we have proved $\sigma(x_i^{k+1}\partial_j)=a_i^{k+1}\th_j$ for any $i,j\in\underline n$, i.e., the assertion also holds for the case $k+1$.

Now we are going to show that $a_i\in A_n$ for all $i\in \underline n$. Since $\th_i\in W_n$, we can assume that $\th_i=\sum_{j\in\underline n}b_{ji}\partial_j$ for some $b_{ji}\in A_n$. Write $a_i=\frac{p_i}{q_i}$ for some coprime polynomials $p_i,q_i\in A_n$. Then noting from  $\sigma(x_i^k\partial_i)=a_i^k\th_i=\sum_{j\in\underline n}\frac{p_i^k}{q_i^k}b_{ji}\partial_i\in W_n$, we obtain that $q_i^k | b_{ji}$ for any $k\in \Z^{\ge0}$. The only possibility for this is that $q_i\in\C^\times$. This shows $a_i\in A_n$.

Next by induction on $r$ we prove $\sigma(x_{i_1}^{k_{i_1}}x_{i_2}^{k_{i_2}}\cdots x_{i_r}^{k_{i_r}}\partial_j)=a_{i_1}^{k_{i_1}}a_{i_2}^{k_{i_2}}\cdots a_{i_r}^{k_{i_r}}\th_j$ for any $j,i_l\in\underline n$ and $k_{i_l}\in\Z^{\ge0}$. By the first paragraph, this statement holds for $r=1$. Suppose this holds for $1\leq r<n$. Without loss of generality, we show that\begin{equation}\label{k_{n-1}}\sigma(x_{1}^{k_{1}}x_{2}^{k_{2}}\cdots x_{n}^{k_{n}}\partial_j)=a_{1}^{k_{1}}a_{2}^{k_{2}}\cdots a_{n}^{k_{n}}\th_j\end{equation} provided that $\sigma(x_{1}^{k_{1}}x_{2}^{k_{2}}\cdots x_{n-1}^{k_{n-1}}\partial_j)=a_{1}^{k_{1}}a_{2}^{k_{2}}\cdots a_{n-1}^{k_{n-1}}\th_j$ holds. By inductive assumption, we have
\begin{eqnarray*}
&&\sigma(x_{1}^{k_{1}}x_{2}^{k_{2}}\cdots x_{n}^{k_{n}}\partial_j)\\ &=&-\frac{(-1)^{\delta_{nj}}}{k_{n+\delta_{nj}-1}+1}\sigma([x_{1}^{k_{1}}x_{2}^{k_{2}}\cdots x_{n-1}^{k_{n-1}+1-\delta_{nj}}\partial_j, x_n^{k_n+\delta_{nj}}\partial_{n+\delta_{nj}-1}])\\ &=&-\frac{(-1)^{\delta_{nj}}}{k_{n+\delta_{nj}-1}+1}[a_{1}^{k_{1}}a_{2}^{k_{2}}\cdots a_{n-1}^{k_{n-1}+1-\delta_{nj}}\th_j, a_n^{k_n+\delta_{nj}}\th_{n+\delta_{nj}-1}]\\
&=&a_{1}^{k_{1}}a_{2}^{k_{2}}\cdots a_{n}^{k_{n}}\th_j.
\end{eqnarray*} That is, the formula \eqref{k_{n-1}} holds. This completes the proof.
\end{proof}

\section{Proof of Theorem \ref{main-theo}}
\setcounter{equation}{0}

Recall that an $n$-tuple $(f_1,f_2,\cdots, f_n)$ of elements in $A_n$ is called a {\em Jacobi tuple} if the Jacobi determination $J(f_1,f_2,\cdots, f_n)={\rm Det}(\partial_j f_i){}_{1\le i,j\le n}\in \C^\times$.

Before beginning to prove Theorem \ref{main-theo}, we also need to present the following result.

\begin{prop}\label{prop--}
Any nonzero endomorphism of $W_n$ is uniquely determined by a Jacobi tuple.
\end{prop}
\begin{proof}
Let $0\neq \sigma\in {\rm End\,}W_n$. Then by Lemmas \ref{lemm--2} and \ref{lemm-homorphism}, there exist $f_i\in A$ and  $\th_i\in W_n$  such that  $$\sigma(x_1^{k_1}x_2^{k_2}\cdots x_n^{k_n}\partial_j)=f_1^{k_1}f_2^{k_2}\cdots f_n^{k_n}\th_j\quad {\rm for\ all\ } j\in\underline n,k_i\in\Z^{\geq 0}$$ and \begin{equation}\label{th_j(f_i)}
\th_j(f_i)=\delta_{ij}\ {\rm for \ any\ } i,j\in \underline n.
\end{equation} For $j\in\underline n$, assume that $\th_j=\sum_{k\in\underline n}a_{jk}\partial_k$ for some $a_{jk}\in A_n$. Then \eqref{th_j(f_i)} is equivalent to $\sum_{k\in\underline n}a_{jk}\partial_k(f_i)=\delta_{ij},$  or in terms of matrix, $$(a_{jk})_{j,k\in\underline n}\ (\partial_k f_i)_{k,i\in\underline n}=I_n\mbox{ (the $n\times n$ identity matrix)}.$$ In particular, $J(f_1,f_2,\cdots, f_n)={\rm Det\,}(\partial_k f_i)_{k,i\in\underline n}\in\C^\times$ and as the inverse  matrix of $(\partial_k f_i)_{k,i\in\underline n}$, $(a_{jk})_{j,k\in\underline n}$ is uniquely determined by $(f_1,f_2,\cdots, f_n)$. This shows that $\sigma$ is uniquely determined by the Jacobi tuple $(f_1,f_2,\cdots, f_n)$.
\end{proof}

\noindent{\em Proof of Theorem \ref{main-theo}}\quad Note that the injectivity and the surjection of $\xi$ follow respectively from  the definition  \eqref{semi-hom1} of $\xi$ and Proposition \ref{prop--},  proving   (2).

To prove (3), we only need to  show that ${\rm Im\,} \zeta\subseteq {\rm Aut\,} W_n$. Since if this is true, then it is easy to see that  $\zeta$ is a bijective map from ${\rm Aut\,}A_n$ onto ${\rm Aut\,}W_n$.   Note that any nonzero element of ${\rm End\,} W_n$ is injective (cf. Section 1).  So it is enough to show that for any given $\tau\in {\rm Aut\,} A_n$, the image $\sigma_{f_\tau}$ of $\tau$ under the map $\zeta$ is surjective in ${\rm End\, }W_n$.   Assume $J(\tau (x_1),\tau (x_2),\cdots, \tau (x_n))=c\in\C^\times$. Let $M^*$ be the adjoint matrix  of $M=\big(\frac{\partial \tau (x_i)}{\partial x_j}\big){}_{1\le i,j\le n}$. Define $\th_j\in W_n$ for $j\in\underline n$ in the following way
$$(\th_1, \th_2,\cdots, \th_n)^T=\frac{1}{c}M^*(\partial_1, \partial_2,\cdots, \partial_n)^T.$$  Here the symbol $T$ stands for the transpose. Then by the definition of $\sigma_{f_\tau}$ (cf. \eqref{sigma-f}), we have \begin{equation}\label{def-sigma}
\sigma_{f_\tau}(h\partial_j)=\tau(h)\th_j\quad {\rm for\ all\ } j\in\underline n{\ \rm and\ } h\in A_n.
\end{equation} Since $M$ is non-degenerate,  so is $M^*$ and thereby each $\partial_i$ is an $A_n$-linear combination of $\th_j$'s, say,
\begin{equation}
\label{eq-partial_i}\partial_{i}=\mbox{$\sum\limits_{j\in\underline n}$}b_{ji}\th_j
\end{equation} for some $b_{ji}\in A_n$. Thus,  $\partial_i=\sigma_{f_\tau}\Big(\sum_{j\in\underline n}\tau^{-1}(b_{ji})\partial_j\Big)\in {\rm Im\,}\sigma_{f_\tau}\subseteq W_n$ for any $i\in\underline n$.

 On the other hand,  by \eqref{def-sigma} and \eqref{eq-partial_i} one can see that
\begin{equation*}
\sigma_{f_\tau}\Big(\tau^{-1}(x_i^2)\mbox{$\sum\limits_{k\in\underline n}$}\tau^{-1}(b_{kj})\partial_k\Big)=x_i^2\mbox{$\sum\limits_{k\in\underline n}$}b_{kj}\th_k=x_i^2\partial_j.
\end{equation*}In particular, $x_i^2\partial_j\in {\rm Im\,}\sigma_{f_\tau}\subseteq W_n$ for any $i,j\in\underline n$.
Thus we have obtained $$\{x_i^2\partial_j,\partial_j\mid i,j\in\underline n\}\subseteq {\rm Im\,}\sigma_{f_\tau}\subseteq W_n.$$ This forces $W_n={\rm Im\,}\sigma_{f_\tau}$ since $\{x_i^2\partial_j,\partial_j\mid i,j\in\underline n\}$ is a generating set of the Lie algebra $W_n$ (recall that we assume $n\ge2$, cf.~Lemma \ref{lemm-homorphism}). This shows the surjection of $\sigma_{f_\tau}$.

As we have mentioned,  the  Jacobi conjecture following from the Witt algebra's conjecture was proved in \cite[Theorem 4.1]{Zhao}, so for (1) it remains to show that the Jacobi conjecture implies the Witt algebra's conjecture. Let $\phi\in {\rm End\,}W_n\bs\{0\}$. We have to show $\phi\in{\rm Aut\,}W_n$. By (2), $\phi$ corresponds to a Jacobi tuple, say, $\xi^{-1}(\phi)=f_\phi=(f_{\phi1},f_{\phi2},\cdots, f_{\phi n}).$ This Jacobi tuple induces an endomorphism $\tau$ of $A_n$ defined by   $$\tau(x_1^{k_1}x_2^{k_2}\cdots x_n^{k_n})=f_{\phi1}^{k_1}f_{\phi2}^{k_2}\cdots f_{\phi n}^{k_n}\quad {\rm for\ any\ } k_l\in\Z^{\geq0}.$$ Now it follows from the equivalent statement of the Jacobi conjecture as  remarked in Section 1 that  $\tau\in {\rm Aut\,} A_n$. So by (3), $\sigma_{f_\tau}=\zeta(\tau)\in {\rm Aut\,}W_n$, where $f_\tau=f_\phi$. Then it follows from (2)  that $\phi=\sigma_{f_\phi}= \sigma_{f_\tau} \in {\rm Aut\,}W_n$, as desired.

Note that (4) follows immediately from (3) and the fact that ${\rm Aut\,}A_2$ is generated by $s=s_1, \,\tau_a\  ({\rm for\ } a\in\C^\times)$ and $\psi_p\ ({\rm for\ } p\in\Z^{\geq0})$ (cf. Section 1 and \cite{N}). This completes the  proof Theorem \ref{main-theo}.

\end{CJK*}
\end{document}